\documentclass[numbers=enddot,12pt,final,onecolumn,notitlepage]{scrartcl}%
\usepackage[headsepline,footsepline,manualmark]{scrlayer-scrpage}
\usepackage[all,cmtip]{xy}
\usepackage{amsfonts}
\usepackage{amssymb}
\usepackage{framed}
\usepackage{amsmath}
\usepackage{comment}
\usepackage{color}
\usepackage{hyperref}
\usepackage[sc]{mathpazo}
\usepackage[T1]{fontenc}
\usepackage{amsthm}
\providecommand{\U}[1]{\protect\rule{.1in}{.1in}}
\theoremstyle{definition}
\newtheorem{theo}{Theorem}[section]
\newenvironment{theorem}[1][]
{\begin{theo}[#1]\begin{leftbar}}
{\end{leftbar}\end{theo}}
\newtheorem{lem}[theo]{Lemma}
\newenvironment{lemma}[1][]
{\begin{lem}[#1]\begin{leftbar}}
{\end{leftbar}\end{lem}}
\newtheorem{prop}[theo]{Proposition}

\newtheorem{defi}[theo]{Definition}

\newtheorem{remk}[theo]{Remark}

\newtheorem{coro}[theo]{Corollary}
\newenvironment{corollary}[1][]
{\begin{coro}[#1]\begin{leftbar}}
{\end{leftbar}\end{coro}}
\newtheorem{conv}[theo]{Convention}

\newtheorem{quest}[theo]{Question}
\newenvironment{question}[1][]
{\begin{quest}[#1]\begin{leftbar}}
{\end{leftbar}\end{quest}}
\newtheorem{exmp}[theo]{Example}

\newenvironment{statement}{\begin{quote}}{\end{quote}}
\newenvironment{verlong}{}{}
\newenvironment{vershort}{}{}

\excludecomment{verlong}
\includecomment{vershort}
\excludecomment{noncompile}

\let\sumnonlimits\sum
\let\prodnonlimits\prod
\renewcommand{\sum}{\sumnonlimits\limits}
\renewcommand{\prod}{\prodnonlimits\limits}
\ihead{A constructive proof of Orzech's theorem}
\ohead{20 April 2026}
\begin{document}
\section*{A constructive proof of Orzech's theorem}

\subsection*{\textit{Darij Grinberg}}

\subsubsection*{version 2.0, 16 April 2026}

\begin{statement}
{\small \textbf{Abstract.} Let $A$ be a commutative ring with unity, and $M$ a
finitely generated $A$-module. In 1971, Morris Orzech showed that any
surjective $A$-module homomorphism from a submodule of $M$ to $M$ must be an
isomorphism. We give a constructive proof of this fact using the
Cayley--Hamilton theorem. }
\end{statement}

The purpose of this note is to prove Morris Orzech's theorem on surjective
homomorphisms of modules \cite[Theorem 1]{orzech} within constructive
mathematics. Our main weapon will be the Cayley--Hamilton theorem.

{\small The LaTeX sourcecode of this note contains additional details of
proofs inside \textquotedblleft verlong\textquotedblright\ environments (i.
e., between \textquotedblleft\texttt{%
$\backslash$%
begin\{verlong\}}\textquotedblright\ and \textquotedblleft\texttt{%
$\backslash$%
end\{verlong\}}\textquotedblright). I doubt they are of any use.}

\begin{center}
***
\end{center}

Let us begin by stating the theorem:

\begin{theorem}
\label{thm.orzech}Let $A$ be a commutative ring with unity. Let $M$ be a
finitely generated $A$-module. Let $N$ be an $A$-submodule of $M$, and let
$f:N\rightarrow M$ be a surjective $A$-module homomorphism. Then, $f$ is an
$A$-module isomorphism.
\end{theorem}

Morris Orzech discovered this fact \cite[Theorem 1]{orzech} in 1971. It
generalizes the following result of Vasconcelos:

\begin{corollary}
\label{cor.orzech}Let $A$ be a commutative ring with unity. Let $M$ be a
finitely generated $A$-module. Let $f:M\rightarrow M$ be a surjective
$A$-module endomorphism of $M$. Then, $f$ is an $A$-module isomorphism.
\end{corollary}

Corollary \ref{cor.orzech} is well-known (e.g., it appears in \cite[Lemma
A.3]{dg-pidisolift} and in \cite{mseold}), but most of its proofs in
literature do not generalize to Theorem \ref{thm.orzech}.

Orzech's original proof of Theorem \ref{thm.orzech} (with the corrections
provided in \cite{mse}, as the original version was shaky) proceeds by
reducing the theorem to the case when $A$ is Noetherian, and then using this
Noetherianness in an elegant and yet mysterious way. The proof is not
constructive and (to my knowledge) cannot easily be made constructive. In this
note, I will present a constructive way to prove Theorem \ref{thm.orzech}.

Let us first make some preparations. We let $\mathbb{N}=\left\{
0,1,2,\ldots\right\}  $. We fix a commutative ring $A$ with unity. For every
$n\in\mathbb{N}$, let $I_{n}$ denote the identity $n\times n$-matrix in
$A^{n\times n}$. We reserve a fresh symbol $X$ as an indeterminate for
polynomials. We embed $A$ into the polynomial ring $A\left[  X\right]  $
canonically, and we use this to embed the matrix ring $A^{n\times n}$ into
$\left(  A\left[  X\right]  \right)  ^{n\times n}$ canonically for every
$n\in\mathbb{N}$. For every $n\in\mathbb{N}$ and any square matrix $M\in
A^{n\times n}$, we define the \textit{characteristic polynomial} $\chi_{M}$ of
$M$ as the polynomial $\det\left(  X\cdot I_{n}-M\right)  $. (This is one of
the two common ways to define a characteristic polynomial of a matrix $M$. The
other way is to define it as $\det\left(  M-X\cdot I_{n}\right)  $. These two
definitions result in two polynomials which differ only by multiplication by
$\left(  -1\right)  ^{n}$.) The famous \textit{Cayley--Hamilton theorem}
states the following:

\begin{theorem}
\label{thm.cayley-hamilton}Let $n\in\mathbb{N}$. Let $A$ be a commutative ring
with unity. Let $M\in A^{n\times n}$. Then, $\chi_{M}\left(  M\right)  =0$.
(In words: Substituting the matrix $M$ for $X$ in the characteristic
polynomial $\chi_{M}$ of $M$ yields the zero matrix.)
\end{theorem}

In this exact form, Theorem \ref{thm.cayley-hamilton} is proven in
\cite{bernhardt}, in \cite[Theorem 3.4]{conrad1} and in \cite[Theorem
2.5]{grinberg-trach}.\footnote{Of course, the notations in these sources don't
exactly match the notations we are using here. For example, the $A$, the $X$
and the $M$ in our Theorem \ref{thm.cayley-hamilton} correspond to the
$\mathbb{K}$, the $t$ and the $A$ in \cite[Theorem 2.5]{grinberg-trach}.} Many
more places contain almost complete proofs of Theorem
\ref{thm.cayley-hamilton}: For example, Theorem \ref{thm.cayley-hamilton} is
proven in most standard texts on linear algebra in the case when $A$ is a
field. Some of these proofs (e.g., the proof given in \cite[Theorem
7.10]{gill}, or the proof given in \cite[Theorem 5.9]{knapp2016}, or the
proofs given in \cite{jerry332}, or Straubing's combinatorial proof given in
\cite{straubing}\footnote{We notice that the two displayed equations right
before the Lemma in \cite[p. 275]{straubing} should be corrected to%
\[
p_{A}^{+}\left(  A\right)  _{ij}=\sum_{\left(  \sigma,\pi\right)  \in
T_{ij}^{+}}\mu\left(  \sigma\right)  \mu\left(  \pi\right)
,\ \ \ \ \ \ \ \ \ \ p_{A}^{-}\left(  A\right)  _{ij}=\sum_{\left(  \sigma
,\pi\right)  \in T_{ij}^{-}}\mu\left(  \sigma\right)  \mu\left(  \pi\right)
.
\]
(To be fair, I do not know if they are wrong in the original printed version
of \cite{straubing} or only in Elsevier's dismal scan of the paper.)} and in
\cite[\S 3]{zeilberger}) can be straightforwardly generalized to the general
case. Even if your favorite proof of Theorem \ref{thm.cayley-hamilton} in the
case when $A$ is a field does not generalize to the general case, it is still
easy to derive the general case from the case of $A$ being a field (this is
what Conrad does in \cite[Theorem 3.4]{conrad1}).

Theorem \ref{thm.cayley-hamilton} has the following direct consequence:

\begin{corollary}
\label{cor.cayley-hamilton}Let $n\in\mathbb{N}$. Let $A$ be a commutative ring
with unity. Let $M\in A^{n\times n}$. Then, there exists an $\left(
n+1\right)  $-tuple $\left(  c_{0},c_{1},\ldots,c_{n}\right)  \in A^{n+1}$
such that $c_{0}M^{0}+c_{1}M^{1}+\cdots+c_{n}M^{n}=0$ and $c_{n}=1$.
\end{corollary}

\begin{verlong}
\begin{proof}
[Proof of Corollary \ref{cor.cayley-hamilton}.]It is well-known that the
characteristic polynomial $\chi_{M}$ of $M$ is a monic polynomial of degree
$n$ over $A$. In other words, there exists an $\left(  n+1\right)  $-tuple
$\left(  d_{0},d_{1},\ldots,d_{n}\right)  \in A^{n+1}$ such that $\chi
_{M}=d_{0}X^{0}+d_{1}X^{1}+\cdots+d_{n}X^{n}$ and $d_{n}=1$. Consider this
$\left(  d_{0},d_{1},\ldots,d_{n}\right)  $. Evaluating both sides of the
equality $\chi_{M}=d_{0}X^{0}+d_{1}X^{1}+\cdots+d_{n}X^{n}$ at $X=M$, we
obtain $\chi_{M}\left(  M\right)  =d_{0}M^{0}+d_{1}M^{1}+\cdots+d_{n}M^{n}$.
Thus, $d_{0}M^{0}+d_{1}M^{1}+\cdots+d_{n}M^{n}=\chi_{M}\left(  M\right)  =0$
(by Theorem \ref{thm.cayley-hamilton}). Hence, there exists an $\left(
n+1\right)  $-tuple $\left(  c_{0},c_{1},\ldots,c_{n}\right)  \in A^{n+1}$
such that $c_{0}M^{0}+c_{1}M^{1}+\cdots+c_{n}M^{n}=0$ and $c_{n}=1$ (namely,
the $\left(  n+1\right)  $-tuple $\left(  d_{0},d_{1},\ldots,d_{n}\right)  $).
This proves Corollary \ref{cor.cayley-hamilton}.
\end{proof}
\end{verlong}

\begin{vershort}
\begin{proof}
[Proof of Corollary \ref{cor.cayley-hamilton}.]It is well-known that the
characteristic polynomial $\chi_{M}$ of $M$ is a monic polynomial of degree
$n$ over $A$. In other words, there exists an $\left(  n+1\right)  $-tuple
$\left(  c_{0},c_{1},\ldots,c_{n}\right)  \in A^{n+1}$ such that $\chi
_{M}=c_{0}X^{0}+c_{1}X^{1}+\cdots+c_{n}X^{n}$ and $c_{n}=1$. Consider this
$\left(  c_{0},c_{1},\ldots,c_{n}\right)  $. Evaluating both sides of the
equality $\chi_{M}=c_{0}X^{0}+c_{1}X^{1}+\cdots+c_{n}X^{n}$ at $X=M$, we
obtain $\chi_{M}\left(  M\right)  =c_{0}M^{0}+c_{1}M^{1}+\cdots+c_{n}M^{n}$.
Thus, $c_{0}M^{0}+c_{1}M^{1}+\cdots+c_{n}M^{n}=\chi_{M}\left(  M\right)  =0$
(by Theorem \ref{thm.cayley-hamilton}). This proves Corollary
\ref{cor.cayley-hamilton}.
\end{proof}
\end{vershort}

We can now use Corollary \ref{cor.cayley-hamilton} to prove the following lemma:

\begin{lemma}
\label{lem.orzech-lem}Let $n\in\mathbb{N}$. Let $g:A^{n}\rightarrow A^{n}$ be
an $A$-linear map. Let $V$ be an $A$-submodule of $A^{n}$ such that
$g^{-1}\left(  V\right)  \subseteq V$. Then, $g\left(  V\right)  \subseteq V$.
\end{lemma}

\begin{proof}
[Proof of Lemma \ref{lem.orzech-lem}.]

\begin{verlong}
If $n=0$, then Lemma \ref{lem.orzech-lem} is obviously
true.\footnote{\textit{Proof.} Assume that $n=0$. Then, $A^{n}=A^{0}$ and thus
$V\subseteq A^{n}=A^{0}=0$, so that $V=0$, and thus $g\left(  \underbrace{V}%
_{=0}\right)  =g\left(  0\right)  =0\subseteq V$. Thus, Lemma
\ref{lem.orzech-lem} is true, qed.} Hence, for the rest of this proof, we can
WLOG assume that we don't have $n=0$. Assume this. We have $n\geq1$ (since we
don't have $n=0$), thus $n-1\in\left\{  0,1,\ldots,n\right\}  $.
\end{verlong}

\begin{vershort}
If $n=0$, then Lemma \ref{lem.orzech-lem} is obviously true (because in this
case, $V\subseteq A^{n}=A^{0}=0$ and thus $V=0$). Hence, for the rest of this
proof, we can WLOG assume that $n\geq1$. Assume this, and notice that this
yields $n-1\in\left\{  0,1,\ldots,n\right\}  $.
\end{vershort}

Let $\left(  e_{1},e_{2},\ldots,e_{n}\right)  $ be the standard basis of the
$A$-module $A^{n}$. (Thus, for every $i\in\left\{  1,2,\ldots,n\right\}  $,
the vector $e_{i}$ is the vector in $A^{n}$ whose $i$-th coordinate is $1$ and
whose other coordinates are all $0$.) Let $M\in A^{n\times n}$ be the $n\times
n$-matrix which represents the $A$-linear map $g:A^{n}\rightarrow A^{n}$ with
respect to this basis $\left(  e_{1},e_{2},\ldots,e_{n}\right)  $ of $A^{n}$.
Then,%
\begin{equation}
Mw=g\left(  w\right)  \ \ \ \ \ \ \ \ \ \ \text{for every }w\in A^{n}.
\label{pf.lem.orzech-lem.Mw}%
\end{equation}

Corollary \ref{cor.cayley-hamilton} shows that there exists an $\left(
n+1\right)  $-tuple $\left(  c_{0},c_{1},\ldots,c_{n}\right)  \in A^{n+1}$
such that $c_{0}M^{0}+c_{1}M^{1}+\cdots+c_{n}M^{n}=0$ and $c_{n}=1$. Consider
this $\left(  c_{0},c_{1},\ldots,c_{n}\right)  $. We have $\sum_{k=0}^{n}%
c_{k}M^{k}=c_{0}M^{0}+c_{1}M^{1}+\cdots+c_{n}M^{n}=0$.

We shall now show that every $u\in\left\{  0,1,\ldots,n\right\}  $ satisfies%
\begin{equation}
\left(  \sum_{k=0}^{n-u}c_{u+k}M^{k}\right)  \left(  V\right)  \subseteq V.
\label{pf.lem.orzech-lem.ind}%
\end{equation}

\textit{Proof of (\ref{pf.lem.orzech-lem.ind}):} We will prove
(\ref{pf.lem.orzech-lem.ind}) by induction over $u$:

\textit{Induction base:} From $n-0=n$, we obtain
\[
\left(  \sum_{k=0}^{n-0} c_{0+k}M^{k}\right)  \left(  V\right)  = \left(
\sum_{k=0}^{n}\underbrace{c_{0+k}}_{=c_{k}}M^{k}\right)  \left(  V\right)
=\underbrace{\left(  \sum_{k=0}^{n}c_{k}M^{k}\right)  }_{=0}\left(  V\right)
=0\left(  V\right)  =0\subseteq V.
\]
In other words, (\ref{pf.lem.orzech-lem.ind}) holds for $u=0$. This completes
the induction base.

\textit{Induction step:} Let $p\in\left\{  0,1,\ldots,n\right\}  $ be such
that $p>0$. Assume that (\ref{pf.lem.orzech-lem.ind}) holds for $u=p-1$. We
now must show that (\ref{pf.lem.orzech-lem.ind}) holds for $u=p$.

We have assumed that (\ref{pf.lem.orzech-lem.ind}) holds for $u=p-1$. In other
words,%
\begin{equation}
\left(  \sum_{k=0}^{n-\left(  p-1\right)  }c_{\left(  p-1\right)  +k}%
M^{k}\right)  \left(  V\right)  \subseteq V.
\label{pf.lem.orzech-lem.ind.pf.IH}%
\end{equation}
Now, $n-\left( p-1\right)  = n-p+1$; therefore,
\begin{align}
\sum_{k=0}^{n-\left(  p-1\right)  } c_{\left(  p-1\right)  +k}M^{k}  &
=\sum_{k=0}^{n-p+1}c_{\left(  p-1\right)  +k}M^{k}=c_{\left(  p-1\right)
+0}M^{0}+\sum_{k=1}^{n-p+1}c_{\left(  p-1\right)  +k}M^{k}\nonumber\\
&  \ \ \ \ \ \ \ \ \ \ \left(  \text{here, we have split off the addend for
}k=0\text{ from the sum}\right) \nonumber\\
&  =\underbrace{c_{\left(  p-1\right)  +0}}_{=c_{p-1}}\underbrace{M^{0}%
}_{=I_{n}}+\sum_{k=0}^{n-p}\underbrace{c_{\left(  p-1\right)  +\left(
k+1\right)  }}_{=c_{p+k}}\underbrace{M^{k+1}}_{=MM^{k}}\nonumber\\
&  \ \ \ \ \ \ \ \ \ \ \left(  \text{here, we have substituted }k+1\text{ for
}k\text{ in the sum}\right) \nonumber\\
&  =c_{p-1}I_{n}+\underbrace{\sum_{k=0}^{n-p}c_{p+k}MM^{k}}_{=M\left(
\sum_{k=0}^{n-p}c_{p+k}M^{k}\right)  }\nonumber\\
& =c_{p-1}I_{n}+M\left(  \sum_{k=0}^{n-p}c_{p+k}M^{k}\right)  .
\label{pf.lem.orzech-lem.ind.pf.1}%
\end{align}
Now, let $v\in V$. Then, applying both sides of the equality
(\ref{pf.lem.orzech-lem.ind.pf.1}) to $v$, we obtain%
\begin{align*}
\left(  \sum_{k=0}^{n-\left(  p-1\right)  }c_{\left(  p-1\right)  +k}%
M^{k}\right)  \left(  v\right)   &  =\left(  c_{p-1}I_{n}+M\left(  \sum
_{k=0}^{n-p}c_{p+k}M^{k}\right)  \right)  v\\
&  =c_{p-1}\underbrace{I_{n}v}_{=v}+\, M\left(  \sum_{k=0}^{n-p}c_{p+k}%
M^{k}\right)  v\\
&  =c_{p-1}v+M\left(  \sum_{k=0}^{n-p}c_{p+k}M^{k}\right)  v.
\end{align*}
Subtracting $c_{p-1}v$ from this equality, we obtain
\[
\left(  \sum_{k=0}^{n-\left(  p-1\right)  }c_{\left(  p-1\right)  +k}%
M^{k}\right)  \left(  v\right)  -c_{p-1}v=M\left(  \sum_{k=0}^{n-p}%
c_{p+k}M^{k}\right)  v=g\left(  \left(  \sum_{k=0}^{n-p}c_{p+k}M^{k}\right)
v\right)
\]
(by (\ref{pf.lem.orzech-lem.Mw}), applied to $w=\left(  \sum_{k=0}%
^{n-p}c_{p+k}M^{k}\right)  v$). Hence,%
\begin{align*}
g\left(  \left(  \sum_{k=0}^{n-p}c_{p+k}M^{k}\right)  v\right)   &  =\left(
\sum_{k=0}^{n-\left(  p-1\right)  }c_{\left(  p-1\right)  +k}M^{k}\right)
\left(  \underbrace{v}_{\in V}\right)  -c_{p-1}\underbrace{v}_{\in V}\\
&  \in\underbrace{\left(  \sum_{k=0}^{n-\left(  p-1\right)  }c_{\left(
p-1\right)  +k}M^{k}\right)  \left(  V\right)  }_{\substack{\subseteq
V\\\text{(by (\ref{pf.lem.orzech-lem.ind.pf.IH}))}}}-\, c_{p-1}V\subseteq
V-c_{p-1}V\subseteq V
\end{align*}
(since $V$ is an $A$-module). Hence, $\left(  \sum_{k=0}^{n-p}c_{p+k}%
M^{k}\right)  v\in g^{-1}\left(  V\right)  \subseteq V$.

Now, let us forget that we fixed $v$. We thus have shown that $\left(
\sum_{k=0}^{n-p}c_{p+k}M^{k}\right)  v\in V$ for every $v\in V$. In other
words, $\left(  \sum_{k=0}^{n-p}c_{p+k}M^{k}\right)  \left(  V\right)
\subseteq V$. In other words, (\ref{pf.lem.orzech-lem.ind}) holds for $u=p$.
This completes the induction step. The induction proof of
(\ref{pf.lem.orzech-lem.ind}) is thus complete.

Now, let us recall that $n-1\in\left\{  0,1,\ldots,n\right\}  $. Hence, we can
apply (\ref{pf.lem.orzech-lem.ind}) to $u=n-1$. As a result, we obtain%
\[
\left(  \sum_{k=0}^{n-\left(  n-1\right)  }c_{\left(  n-1\right)  +k}%
M^{k}\right)  \left(  V\right)  \subseteq V.
\]
Since%
\[
\underbrace{\sum_{k=0}^{n-\left(  n-1\right)  }}_{=\sum_{k=0}^{1}}c_{\left(
n-1\right)  +k}M^{k}=\sum_{k=0}^{1}c_{\left(  n-1\right)  +k}M^{k}%
=\underbrace{c_{\left(  n-1\right)  +0}}_{=c_{n-1}}\underbrace{M^{0}}_{=I_{n}%
}+\underbrace{c_{\left(  n-1\right)  +1}}_{=c_{n}=1}\underbrace{M^{1}}%
_{=M}=c_{n-1}I_{n}+M,
\]
this rewrites as $\left(  c_{n-1}I_{n}+M\right)  \left(  V\right)  \subseteq
V$. Now, let $w\in V$. Then,%
\[
\left(  c_{n-1}I_{n}+M\right)  \left(  \underbrace{w}_{\in V}\right)
\in\left(  c_{n-1}I_{n}+M\right)  \left(  V\right)  \subseteq V.
\]
Since $\left(  c_{n-1}I_{n}+M\right)  \left(  w\right)  =c_{n-1}%
\underbrace{I_{n}w}_{=w}+\underbrace{Mw}_{\substack{=g\left(  w\right)
\\\text{(by (\ref{pf.lem.orzech-lem.Mw}))}}}=c_{n-1}w+g\left(  w\right)  $,
this rewrites as $c_{n-1}w+g\left(  w\right)  \in V$. Hence,%
\[
g\left(  w\right)  \in V-c_{n-1}\underbrace{w}_{\in V}\subseteq V-c_{n-1}%
V\subseteq V\ \ \ \ \ \ \ \ \ \ \left(  \text{since }V\text{ is an
}A\text{-module}\right)  .
\]

Now, let us forget that we fixed $w$. We thus have shown that $g\left(
w\right)  \in V$ for every $w\in V$. In other words, $g\left(  V\right)
\subseteq V$. This proves Lemma \ref{lem.orzech-lem}.
\end{proof}

Our next step is a proof of Theorem \ref{thm.orzech} in the case when $N$
(rather than $M$) is finitely generated:

\begin{lemma}
\label{lem.orzech.fingen}Let $A$ be a commutative ring with unity. Let $M$ be
an $A$-module. Let $N$ be an $A$-submodule of $M$ such that $N$ is finitely
generated as an $A$-module. Let $f:N\rightarrow M$ be a surjective $A$-module
homomorphism. Then, $f$ is an $A$-module isomorphism.
\end{lemma}

\begin{proof}
[Proof of Lemma \ref{lem.orzech.fingen}.]We know that $N$ is finitely
generated. In other words, there exist finitely many elements $a_{1}%
,a_{2},\ldots,a_{n}$ of $N$ such that $N$ is generated by $a_{1},a_{2}%
,\ldots,a_{n}$ as an $A$-module. Consider these $a_{1},a_{2},\ldots,a_{n}$.

Let $\left(  e_{1},e_{2},\ldots,e_{n}\right)  $ be the standard basis of the
$A$-module $A^{n}$. (Thus, for every $i\in\left\{  1,2,\ldots,n\right\}  $,
the vector $e_{i}$ is the vector in $A^{n}$ whose $i$-th coordinate is $1$ and
whose other coordinates are all $0$.) Clearly, in order to define an
$A$-linear map from $A^{n}$ to an $A$-module, it is enough to specify the
images of this map at the basis vectors $e_{i}$ (and these images can be
chosen arbitrarily). Thus, we can define an $A$-linear map $p:A^{n}\rightarrow
N$ by%
\[
\left(  p\left(  e_{i}\right)  =a_{i}\ \ \ \ \ \ \ \ \ \ \text{for every }%
i\in\left\{  1,2,\ldots,n\right\}  \right)  .
\]
Consider this $p$.

\begin{verlong}
It is fairly clear that the map $p$ is surjective\footnote{\textit{Proof.} The
$A$-module $N$ is generated by $a_{1},a_{2},\ldots,a_{n}$. In other words,
$N=Aa_{1}+Aa_{2}+\cdots+Aa_{n}$. Now, for every $i\in\left\{  1,2,\ldots
,n\right\}  $, we have $A\underbrace{p\left(  e_{i}\right)  }_{=a_{i}}=Aa_{i}%
$. Adding together these equalities for all $i\in\left\{  1,2,\ldots
,n\right\}  $, we obtain $Ap\left(  e_{1}\right)  +Ap\left(  e_{2}\right)
+\cdots+Ap\left(  e_{n}\right)  =Aa_{1}+Aa_{2}+\cdots+Aa_{n}=N$. Now,
\begin{align*}
p\left(  Ae_{1}+Ae_{2}+\cdots+Ae_{n}\right)   &  =Ap\left(  e_{1}\right)
+Ap\left(  e_{2}\right)  +\cdots+Ap\left(  e_{n}\right)
\ \ \ \ \ \ \ \ \ \ \left(  \text{since the map }p\text{ is }A\text{-linear}%
\right) \\
&  =N,
\end{align*}
so that $N=p\left(  \underbrace{Ae_{1}+Ae_{2}+\cdots+Ae_{n}}_{\subseteq A^{n}%
}\right)  \subseteq p\left(  A^{n}\right)  $. In other words, the map $p$ is
surjective, qed.}. The composition of two surjective maps is always
surjective. Thus, the composition $f\circ p:A^{n}\rightarrow M$ of the maps
$f$ and $p$ is surjective (since the maps $f$ and $p$ are surjective). Thus,
$M=\left(  f\circ p\right)  \left(  A^{n}\right)  $.
\end{verlong}

\begin{vershort}
The generators $a_{1},a_{2},\ldots,a_{n}$ of the $A$-module $N$ are in the
image of the map $p$ (since $a_{i}=p\left(  e_{i}\right)  $ for every
$i\in\left\{  1,2,\ldots,n\right\}  $). Thus, the $A$-linear map
$p:A^{n}\rightarrow N$ is surjective. Hence, the map $f\circ p:A^{n}%
\rightarrow M$ is also surjective (being the composition of the surjective
maps $f$ and $p$). Hence, $M=\left(  f\circ p\right)  \left(  A^{n}\right)  $.
\end{vershort}

Let us now define $n$ elements $h_{1},h_{2},\ldots,h_{n}$ of $A^{n}$ as
follows: For every $i\in\left\{  1,2,\ldots,n\right\}  $, there exists a
vector $h\in A^{n}$ such that $p\left(  e_{i}\right)  =\left(  f\circ
p\right)  \left(  h\right)  $ (since $p\left(  e_{i}\right)  \in N\subseteq
M=\left(  f\circ p\right)  \left(  A^{n}\right)  $). Pick such an $h$ and
denote it by $h_{i}$. Thus, for every $i\in\left\{  1,2,\ldots,n\right\}  $,
we have defined a vector $h_{i}\in A^{n}$ such that
\begin{equation}
p\left(  e_{i}\right)  =\left(  f\circ p\right)  \left(  h_{i}\right)  .
\label{pf.lem.orzech.fingen.hi}%
\end{equation}
We have thus constructed $n$ elements $h_{1},h_{2},\ldots,h_{n}$ of $A^{n}$.

Recall that, in order to define an $A$-linear map from $A^{n}$ to an
$A$-module, it is enough to specify the images of this map at the basis
vectors $e_{i}$ (and these images can be chosen arbitrarily). Hence, we can
define an $A$-linear map $g:A^{n}\rightarrow A^{n}$ by%
\[
\left(  g\left(  e_{i}\right)  =h_{i}\ \ \ \ \ \ \ \ \ \ \text{for every }%
i\in\left\{  1,2,\ldots,n\right\}  \right)  .
\]
Consider this $g$. Then, $f\circ p\circ g=p$\ \ \ \ \footnote{\textit{Proof.}
Every $i\in\left\{  1,2,\ldots,n\right\}  $ satisfies%
\[
\left(  f\circ p\circ g\right)  \left(  e_{i}\right)  =\left(  f\circ
p\right)  \left(  \underbrace{g\left(  e_{i}\right)  }_{=h_{i}}\right)
=\left(  f\circ p\right)  \left(  h_{i}\right)  =p\left(  e_{i}\right)
\ \ \ \ \ \ \ \ \ \ \left(  \text{by (\ref{pf.lem.orzech.fingen.hi})}\right)
.
\]
In other words, the $A$-linear maps $f\circ p\circ g$ and $p$ are equal to
each other on each element of the basis $\left(  e_{1},e_{2},\ldots
,e_{n}\right)  $ of $A^{n}$. Consequently, these maps $f\circ p\circ g$ and
$p$ must be identical (because if two $A$-linear maps from some $A$-module $P$
are equal to each other on each element of a given basis of $P$, then these
two maps must be identical). In other words, $f\circ p\circ g=p$, qed.}.

Let $V$ be the $A$-submodule $\operatorname*{Ker}\left(  f\circ p\right)  $ of
$A^{n}$. It is straightforward to prove that $g^{-1}\left(  V\right)
\subseteq V$\ \ \ \ \footnote{\textit{Proof.} Let $w\in g^{-1}\left(
V\right)  $. Then, $g\left(  w\right)  \in V=\operatorname*{Ker}\left(  f\circ
p\right)  $, so that $\left(  f\circ p\right)  \left(  g\left(  w\right)
\right)  =0$. Thus, $0=\left(  f\circ p\right)  \left(  g\left(  w\right)
\right)  =\underbrace{\left(  f\circ p\circ g\right)  }_{=p}\left(  w\right)
=p\left(  w\right)  $, so that $p\left(  w\right)  =0$ and thus $\left(
f\circ p\right)  \left(  w\right)  =f\left(  \underbrace{p\left(  w\right)
}_{=0}\right)  =f\left(  0\right)  =0$ (since $f$ is $A$-linear).
Consequently, $w\in\operatorname*{Ker}\left(  f\circ p\right)  =V$.
\par
Let us now forget that we fixed $w$. We thus have shown that $w\in V$ for
every $w\in g^{-1}\left(  V\right)  $. In other words, $g^{-1}\left(
V\right)  \subseteq V$, qed.}. Lemma \ref{lem.orzech-lem} thus shows that
$g\left(  V\right)  \subseteq V$.

Let now $w\in\operatorname*{Ker}f$ be arbitrary. Then, $w\in N$ satisfies
$f\left(  w\right)  =0$ (since $w\in\operatorname*{Ker}f$). But the map $p$ is
surjective; thus, $N=p\left(  A^{n}\right)  $. Hence, $w\in N=p\left(
A^{n}\right)  $. In other words, there exists some $v\in A^{n}$ such that
$w=p\left(  v\right)  $. Consider this $v$. We have $\left(  f\circ p\right)
\left(  v\right)  =f\left(  \underbrace{p\left(  v\right)  }_{=w}\right)
=f\left(  w\right)  =0$, so that $v\in\operatorname*{Ker}\left(  f\circ
p\right)  =V$ and thus $g\left(  \underbrace{v}_{\in V}\right)  \in g\left(
V\right)  \subseteq V=\operatorname*{Ker}\left(  f\circ p\right)  $ and thus
$\left(  f\circ p\right)  \left(  g\left(  v\right)  \right)  =0$. Thus,
$\left(  f\circ p\circ g\right)  \left(  v\right)  =\left(  f\circ p\right)
\left(  g\left(  v\right)  \right)  =0$. Since $f\circ p\circ g=p$, this
rewrites as $p\left(  v\right)  =0$. Thus, $w=p\left(  v\right)  =0$.

Now, let us forget that we fixed $w$. We thus have proven that $w=0$ for every
$w\in\operatorname*{Ker}f$. In other words, $\operatorname*{Ker}f=0$. Hence,
the map $f$ is injective. Since $f$ is also surjective, this yields that $f$
is bijective. Thus, $f$ is an $A$-module isomorphism (since $f$ is an
$A$-module homomorphism). This proves Lemma \ref{lem.orzech.fingen}.
\end{proof}

Now, we can finally step to the proof of Theorem \ref{thm.orzech}:

\begin{proof}
[Proof of Theorem \ref{thm.orzech}.]We know that $M$ is finitely generated. In
other words, there exist finitely many elements $a_{1},a_{2},\ldots,a_{n}$ of
$M$ such that $M$ is generated by $a_{1},a_{2},\ldots,a_{n}$ as an $A$-module.
Consider these $a_{1},a_{2},\ldots,a_{n}$.

Notice that $M=f\left(  N\right)  $ (since the map $f$ is surjective).

For every $i\in\left\{  1,2,\ldots,n\right\}  $, we define an element $g_{i}$
of $N$ as follows: There exists some $g\in N$ such that $a_{i}=f\left(
g\right)  $ (since $a_{i}\in M=f\left(  N\right)  $). Pick such a $g$ and
denote it by $g_{i}$. Thus, for every $i\in\left\{  1,2,\ldots,n\right\}  $,
we have defined some $g_{i}\in N$ satisfying%
\begin{equation}
a_{i}=f\left(  g_{i}\right)  . \label{pf.thm.orzech.gi}%
\end{equation}
Hence, we have defined $n$ elements $g_{1},g_{2},\ldots,g_{n}$ of $N$.

\begin{verlong}
Let $v\in\operatorname*{Ker}f$. Let $N^{\prime}$ be the $A$-submodule
$Av+\left(  Ag_{1}+Ag_{2}+\cdots+Ag_{n}\right)  $ of $N$. Then,%
\[
N^{\prime}=Av+\left(  Ag_{1}+Ag_{2}+\cdots+Ag_{n}\right)  =Av+Ag_{1}%
+Ag_{2}+\cdots+Ag_{n}.
\]
In other words, the $A$-module $N^{\prime}$ is generated by the $n+1$ elements
$v,g_{1},g_{2},\ldots,g_{n}$. Hence, the $A$-module $N^{\prime}$ is finitely
generated. Also, $N^{\prime}\subseteq N\subseteq M$. It is easy to see that
the $A$-linear map $f\mid_{N^{\prime}}:N^{\prime}\rightarrow M$ is
surjective\footnote{\textit{Proof.} The $A$-module $M$ is generated by
$a_{1},a_{2},\ldots,a_{n}$. In other words, $M=Aa_{1}+Aa_{2}+\cdots+Aa_{n}$.
Now, for every $i\in\left\{  1,2,\ldots,n\right\}  $, we have
$A\underbrace{f\left(  g_{i}\right)  }_{\substack{=a_{i}\\\text{(by
(\ref{pf.thm.orzech.gi}))}}}=Aa_{i}$. Adding together these equalities for all
$i\in\left\{  1,2,\ldots,n\right\}  $, we obtain $Af\left(  g_{1}\right)
+Af\left(  g_{2}\right)  +\cdots+Af\left(  g_{n}\right)  =Aa_{1}+Aa_{2}%
+\cdots+Aa_{n}=M$. Now, $Ag_{1}+Ag_{2}+\cdots+Ag_{n}$ is an $A$-submodule of
$N^{\prime}$ (since $N^{\prime}=Av+\left(  Ag_{1}+Ag_{2}+\cdots+Ag_{n}\right)
$), and we have%
\begin{align*}
\left(  f\mid_{N^{\prime}}\right)  \left(  Ag_{1}+Ag_{2}+\cdots+Ag_{n}\right)
&  =f\left(  Ag_{1}+Ag_{2}+\cdots+Ag_{n}\right) \\
&  =Af\left(  g_{1}\right)  +Af\left(  g_{2}\right)  +\cdots+Af\left(
g_{n}\right)  \ \ \ \ \ \ \ \ \ \ \left(  \text{since the map }f\text{ is
}A\text{-linear}\right) \\
&  =M,
\end{align*}
so that $M=\left(  f\mid_{N^{\prime}}\right)  \left(  \underbrace{Ag_{1}%
+Ag_{2}+\cdots+Ag_{n}}_{\subseteq N^{\prime}}\right)  \subseteq\left(
f\mid_{N^{\prime}}\right)  \left(  N^{\prime}\right)  $. In other words, the
map $f\mid_{N^{\prime}}$ is surjective, qed.}. Hence, Lemma
\ref{lem.orzech.fingen} (applied to $N^{\prime}$ and $f\mid_{N^{\prime}}$
instead of $N$ and $f$) yields that $f\mid_{N^{\prime}}$ is an $A$-module
isomorphism. In particular, $f\mid_{N^{\prime}}$ is injective. Thus,
$\operatorname*{Ker}\left(  f\mid_{N^{\prime}}\right)  =0$.
\end{verlong}

\begin{vershort}
Let $v\in\operatorname*{Ker}f$. We shall prove that $v=0$.

Let $N^{\prime}$ be the $A$-submodule $Av+\left(  Ag_{1}+Ag_{2}+\cdots
+Ag_{n}\right)  $ of $N$. Then, the $A$-module $N^{\prime}$ is finitely
generated (in fact, it is generated by the $n+1$ elements $v,g_{1}%
,g_{2},\ldots,g_{n}$) and satisfies $N^{\prime}\subseteq N\subseteq M$. Also,
the $A$-linear map $f\mid_{N^{\prime}}:N^{\prime}\rightarrow M$ is surjective,
because its image contains the generators $a_{1},a_{2},\ldots,a_{n}$ of $M$
(in fact, for every $i\in\left\{  1,2,\ldots,n\right\}  $, we have $g_{i}\in
Ag_{i}\subseteq Av+\left(  Ag_{1}+Ag_{2}+\cdots+Ag_{n}\right)  \subseteq
N^{\prime}$ and thus $a_{i}=f\left(  \underbrace{g_{i}}_{\in N^{\prime}%
}\right)  =\left(  f\mid_{N^{\prime}}\right)  \left(  g_{i}\right)  $, which
shows that the image of $f\mid_{N^{\prime}}$ contains $a_{i}$). Hence, Lemma
\ref{lem.orzech.fingen} (applied to $N^{\prime}$ and $f\mid_{N^{\prime}}$
instead of $N$ and $f$) yields that $f\mid_{N^{\prime}}$ is an $A$-module
isomorphism. In particular, $f\mid_{N^{\prime}}$ is injective. Thus,
$\operatorname*{Ker}\left(  f\mid_{N^{\prime}}\right)  =0$.
\end{vershort}

But $v\in Av\subseteq Av+\left(  Ag_{1}+Ag_{2}+\cdots+Ag_{n}\right)
=N^{\prime}$ and $\left(  f\mid_{N^{\prime}}\right)  \left(  v\right)
=f\left(  v\right)  =0$ (since $v\in\operatorname*{Ker}f$). Hence,
$v\in\operatorname*{Ker}\left(  f\mid_{N^{\prime}}\right)  =0$. In other
words, $v=0$.

Now, let us forget that we fixed $v$. We thus have shown that $v=0$ for every
$v\in\operatorname*{Ker}f$. In other words, $\operatorname*{Ker}f=0$. Hence,
the map $f$ is injective. Since $f$ is also surjective, this yields that $f$
is bijective. Thus, $f$ is an $A$-module isomorphism (since $f$ is an
$A$-module homomorphism). This proves Theorem \ref{thm.orzech}.
\end{proof}

\begin{proof}
[Proof of Corollary \ref{cor.orzech}.]Corollary \ref{cor.orzech} follows
immediately from Theorem \ref{thm.orzech} (applied to $N=M$).
\end{proof}


We end with two questions. Corollary~\ref{cor.orzech} is known to have a
multiplicative version (due to Orzech and Ribes \cite[Theorem 6]%
{orzrib}\footnote{See also \cite{mse1217452} and \cite[Example 4]{sharifi1}}):
If $B$ is a finitely generated commutative $A$-algebra, and $f:B\rightarrow B$
is a surjective $A$-algebra endomorphism of $B$, then $f$ is an $A$-algebra isomorphism.

\begin{question}
Does this have a constructive proof?
\end{question}

\begin{question}
Is there an analogous variant of Orzech's theorem (Theorem~\ref{thm.orzech})
for $A$-algebras instead of $A$-modules?
\end{question}


\begin{thebibliography}{99}                                                                                               %


\bibitem {orzech}Morris Orzech, \textit{Onto Endomorphisms are Isomorphisms},
Amer. Math. Monthly 78 (1971), pp. 357--362.\newline\url{https://doi.org/10.1080/00029890.1971.11992759}

\bibitem {mse}\textit{Is Orzech's generalization of the
surjective-endomorphism-is-injective theorem correct?}, math.stackexchange
question \#1065786 and consequent discussion.\newline\url{http://math.stackexchange.com/questions/1065786}

\bibitem {mseold}\textit{Surjective endomorphisms of finitely generated
modules are isomorphisms}, math.stackexchange question \#239364 and consequent
discussion.\newline\url{http://math.stackexchange.com/questions/239364}

\bibitem {mse1217452}\textit{Is a surjective $R$-endomorphism over a finitely
generated $R$-algebra always bijective?}, math.stackexchange question
\#1217452 and consequent discussion.\newline\url{https://math.stackexchange.com/questions/1217452}

\bibitem {orzrib}Morris Orzech, Luis Ribes, \textit{Residual finiteness and
the Hopf property in rings}, Journal of Algebra \textbf{15} (1970), Issue 1,
pp. 81--88.\newline\url{https://doi.org/10.1016/0021-8693(70)90087-6}

\bibitem {gill}Joel G. Broida and S. Gill Williamson, \textit{Comprehensive
Introduction to Linear Algebra}, Addison-Wesley 1989.\newline\url{http://cseweb.ucsd.edu/~gill/CILASite/}

\bibitem {grinberg-trach}Darij Grinberg, \textit{The trace Cayley-Hamilton
theorem}, 14 April 2026,
\href{https://arxiv.org/abs/2510.20689v2}{arXiv:2510.20689v2}.\newline\url{http://www.cip.ifi.lmu.de/~grinberg/algebra/trach.pdf}

\bibitem {straubing}Howard Straubing, \textit{A combinatorial proof of the
Cayley-Hamilton theorem}, Discrete Mathematics, Volume 43, Issues 2--3, 1983,
pp. 273--279.\newline\url{https://doi.org/10.1016/0012-365X(83)90164-4}

\bibitem {zeilberger}Doron Zeilberger, \textit{A combinatorial approach to
matrix algebra}, Discrete Mathematics, Volume 56, Issue 1, September 1985, pp.
61--72.\newline\url{https://doi.org/10.1016/0012-365X(85)90192-X} \newline A
corrected version is available at\newline\url{https://sites.math.rutgers.edu/~zeilberg/mamarimY/DM85dg.pdf}

\bibitem {bernhardt}Chris Bernhardt, \textit{A proof of the Cayley Hamilton
theorem}.\newline\url{http://web.archive.org/web/20180417053920/http://faculty.fairfield.edu/cbernhardt/cayleyhamilton.pdf}

\bibitem {jerry332}Jerry Shurman, \textit{The Cayley-Hamilton theorem via
multilinear algebra}, version 24 May 2015.\newline\url{http://people.reed.edu/~jerry/332/28ch.pdf}

\bibitem {knapp2016}Anthony W. Knapp, \textit{Basic Algebra}, Digital Second
Edition, 2016.\newline\url{http://www.math.stonybrook.edu/~aknapp/download.html}

\bibitem {conrad1}Keith Conrad, \textit{Universal Identities I}, 2013.\newline\url{http://www.math.uconn.edu/~kconrad/blurbs/}

\bibitem {dg-pidisolift}Darij Grinberg, \textit{A note on lifting isomorphisms
of modules over PIDs}.\newline\url{http://www.cip.ifi.lmu.de/~grinberg/algebra/pidisolift.pdf}

\bibitem {sharifi1}Yaghoub Sharifi, \textit{Abstract Algebra: Hopfian algebras
(1)}, 18 December 2021, blog post.\newline\url{https://ysharifi.wordpress.com/2021/12/18/hopfian-algebras-1/}
\end{thebibliography}
\end{document}